\documentclass[12pt]{article}
\usepackage{amssymb,amsthm, amsmath}
\usepackage{graphicx}
\usepackage{epsfig}

\newtheorem{theorem}{Theorem}[section]
\newtheorem{corollary}[theorem]{Corollary}
\newtheorem{lemma}[theorem]{Lemma}
\theoremstyle{definition}
\newtheorem{definition}[theorem]{Definition}

\newcommand{\RNum}[1]{\uppercase\expandafter{\romannumeral #1\relax}}

\begin{document}
\title{An improved upper bound for the Erd\H{o}s-Szekeres conjecture}
\author{Hossein Nassajian Mojarrad\thanks{EPFL, Lausanne. Partially supported by Swiss National Science Foundation Grants 200020-144531, 200021-137574, 200020-162884 and 200020-144531.} \and Georgios Vlachos \thanks{MIT, Cambridge.}
}
\date{}
\maketitle

\begin{abstract}
Let $ES(n)$ denote the minimum natural number such that every set of $ES(n)$ points in general position in the plane contains $n$ points in convex position. In 1935,  Erd\H{o}s and Szekeres proved that $ES(n) \le {2n-4 \choose n-2}+1$. In 1961, they obtained the lower bound $2^{n-2}+1 \le ES(n)$, which they conjectured to be optimal. In this paper, we prove that $$ES(n) \le {2n-5 \choose n-2}-{2n-8 \choose n-3}+2 \approx \frac{7}{16} {2n-4 \choose n-2}.$$  

\end{abstract}


\section{Introduction} 

In 1933, Esther Klein asked about the existence of the least natural number $ES(n)$ such that every set of $ES(n)$ points in general position in the plane contains $n$ points in convex position. Erd\H{o}s and Szekeres \cite{ES35} gave a positive answer in 1935, by proving that $ES(n) \le {2n-4 \choose n-2}+1$. In 1961, they gave a construction in \cite{ES61} and proved that $2^{n-2}+1 \le ES(n)$. The lower bound is conjectured to be tight.

In 1998, Chung and Graham \cite{CG} improved the upper bound by $1$. In the same year, Kleitman and Pachter \cite{KP} proved that $ES(n) \le {2n-4 \choose n-2}-2n+7$. Shortly after that, T\'{o}th and Valtr \cite{TV98} improved the bound roughly by a factor of $2$ by showing that $ES(n) \le {2n-5 \choose n-2}+2$. In \cite{TV03}, they combined the ideas from \cite{TV98} and \cite{CG} to improve this bound by another $1$. We refer the reader to \cite{BMP} for other questions and results related to the Erd\H{o}s-Szekeres theorem.

In his recent paper \cite{V}, Vlachos further improved the bound and showed that
\begin{equation} 
\tag{1}
ES(n) \le {2n-5 \choose n-2}-{2n-8 \choose n-3}+{2n-10 \choose n-3}+2,
\label{eq1}
\end{equation}
which implies 
\begin{equation} 
\tag{2}
\limsup\limits_{n\rightarrow\infty} \frac{ES(n)}{{2n-5 \choose n-2}} \le \frac{29}{32}.
\label{eq2}
\end{equation}
\\
Vlachos' manuscript \cite{V} has revitalized the subject and has led to further improvements. Using slightly different techniques, Norin and Yuditsky \cite{NY} showed that 
\begin{equation} 
\tag{2$\mathrm{'}$}
\limsup\limits_{n\rightarrow\infty} \frac{ES(n)}{{2n-5 \choose n-2}} \le \frac{7}{8}.
\label{eq3}
\end{equation}

On the other hand, during the refereeing process of \cite{V}, each of the two current authors independently fine-tuned the original arguments of \cite{V} to get rid of the term ${2n-10 \choose n-3}$ in \eqref{eq1}.
\\Our main result is the following.
\begin{theorem} \label{main}
Let $ES(n)$ be the minimum natural number such that every set of $ES(n)$ points in the plane contains $n$ points in convex position. For any natural number $n \ge 6$, we have 
$$ES(n) \le {2n-5 \choose n-2}-{2n-8 \choose n-3}+2.$$
\end{theorem}
Note that this result is slightly stronger than that in \cite{NY}, but it yields the same asymptotic upper bound \eqref{eq3}.

In Section 2 of the paper, a partitioning of any point set, with some nice applications is introduced. We use this partitioning to give an alternative proof of the first upper bound for $ES(n)$, which was obtained by Erd\H{o}s and Szekeres \cite{ES35}. In Section 3, we introduce an auxiliary function called \emph{the convexification function} $h(m,l)$, and establish some of its most important properties. Among them, Lemma \ref{subadd} and Theorem \ref{m,4} are taken from \cite{V}. We prove another property of $h(m,l)$ in Theorem \ref{4,l}. Combining this result with the other two will give us an improved upper bound on the convexification function. By using this bound and applying a projective transformation taken from \cite{TV98}, we prove Theorem \ref{main} in the last section, and deduce the asymptotic bound \eqref{eq3} from it.


\section{Preliminary results}

First, we define a few properties of the point sets in the plane.
\begin{definition} \label{gpv}  A set is said to be \emph{in general position}, if it does not contain any three collinear points. 
\\ A set is \emph{non-vertical}, if any vertical line in the plane contains at most one point of the set. In other words, no line spanned by the points of the set is vertical.
\end{definition}
Unless otherwise stated, all sets we consider in this paper (even when considering the union of two sets) are assumed to be \emph{finite}, \emph{non-vertical} and \emph{in general position}. 

\begin{definition} [Cups and Caps]
Let $m \ge 3$ be a natural number. We say that the points $(x_1,y_1),(x_2,y_2),\dots,(x_m,y_m)$ form an \emph{$m$-cup} if they satisfy the following properties:
\begin{itemize}
\item[1.] $x_1 < x_2 < \dots < x_m$;
\item[2.] $\frac {y_2-y_1} {x_2-x_1} < \frac {y_3-y_2} {x_3-x_2} < \dots < \frac {y_m-y_{m-1}} {x_m-x_{m-1}}$.
\end{itemize} 
In a similar way, they form an \emph{$m$-cap} if we have
\begin{itemize}
\item[1.] $x_1 < x_2 < \dots < x_m$;
\item[2.] $\frac {y_2-y_1} {x_2-x_1} > \frac {y_3-y_2} {x_3-x_2} > \dots > \frac {y_m-y_{m-1}} {x_m-x_{m-1}}$.
\end{itemize} 
\end{definition}
\begin{figure} [ht]
\begin{center}
\epsfbox{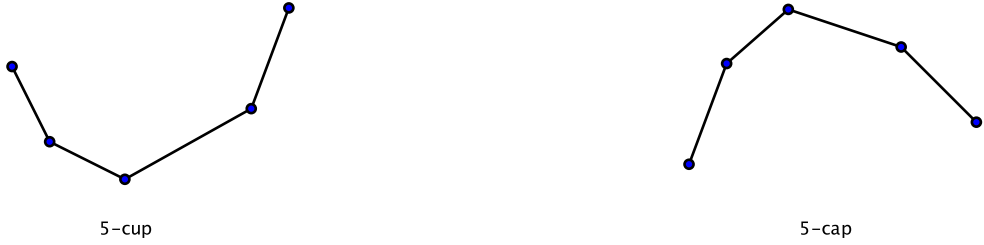}
\end{center}
\caption{A $5$-cup and a $5$-cap.}
\label{cup-cap}
\end{figure}

The following definition introduces a partitioning of any point set in the plane with some nice applications (Lemma \ref{ext}). For any point $p$, let $(p_x,p_y)$ denote its coordinates.
\begin{definition}[Upper and Lower subsets] \label{UL}
Consider the point set $S$ in the plane. Pick an arbitrary $s=(s_x,s_y) \in S$ and let 
$$S_s^-=\{s' \in S : s'_x < s_x\}, S_s^+=S\setminus (S_s^-\cup\{s\}).$$ 
For any $s' \in S\setminus\{s\}$, define the angle between $s$ and $s'$, $\angle(s,s')$, in the following way.
\begin{itemize}
\item[1.] If $s' \in S_s^-$: $\angle(s,s')=\angle(s',s,(s_x,s_y+1))$;
\item[2.] If $s' \in S_s^+$: $\angle(s,s')=\angle((s_x,s_y-1),s,s')$.
\end{itemize} 
Let $\angle(s,p_s)=\min\{\angle(s,s'): s' \in S\setminus\{s\}\}$. Note that since $S$ is in general position, $p_s$ is the unique point of $S\setminus\{s\}$ with that property. Now if $p_s \in S_s^+$, then we call $s$ an upper point for $S$; otherwise we call it a lower point. See Figure \ref{ULfig}. We denote the \emph{upper and lower subsets} of $S$ by $U_S$ and $L_S$, which consist of the upper and lower points of $S$, respectively. 

\end{definition}

\begin{figure} [h]
\begin{center}
\epsfbox{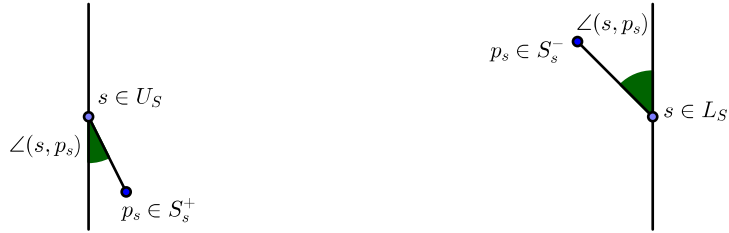}
\end{center}
\caption{Upper and lower points.}
\label{ULfig}
\end{figure}

\begin{lemma} \label{part}
For any point set $S$, $\{U_S,L_S\}$ gives a nontrivial partition of $S$.
\end{lemma}
\begin{proof}
According to the Definition \ref{UL}, it is easy to see that $S=U_S \cup L_S$ and $U_S\cap L_S=\emptyset$. So we know that at least one of $U_S$ and $L_S$ is nonempty. Without loss of generality, suppose there exists an $s \in U_S$. Consider the point $p_s \in S\setminus\{s\}$ forming the minimum angle with $s$. It is easy to check that $p_s \in L_S$, since otherwise we get $\angle(s,p_{p_s}) < \angle(s,p_s)$, which is in contradiction with the choice of $p_s$. Therefore, we also have $L_S\neq \emptyset$.
\end{proof}

\begin{lemma} \label{ext}
Consider the disjoint point sets $B, S$ and let $m,l \ge 3$ be natural numbers. The following statements hold.
\begin{itemize}
\item[1.] Any $l$-cap in $B\cup S$ with the rightmost point in $U_S$ and the second rightmost point in $S$ can be extended to an $(l+1)$-cap, by adding an appropriate point of $L_S$ to its right.
\item[2.] Any $m$-cup in $B\cup S$ with the leftmost point in $L_S$ and the second leftmost point in $S$ can be extended to an $(m+1)$-cup, by adding an appropriate point of $U_S$ to its left.
\end{itemize}
\end{lemma}
\begin{proof} We give the sketch of the proof for the first statement; the second one can be proved in a similar way. Denote the $l$-cap by $v_1v_2\dots v_l$, where $v_{l-1} \in S$ and $v_l \in U_S$. As described in the Definition \ref{UL}, consider the point $p_{v_l} \in S\setminus\{v_l\}$ forming the minimum angle with $v_l$. By the proof of Lemma \ref{part}, we have $p_{v_l} \in L_S$. All the points of $S$ lying to the left of $v_l$ are below the line $v_lp_{v_l}$; in particular $v_{l-1}$ is below  $v_lp_{v_l}$. See Figure \ref{lemmaext}. So we get that $v_1v_2\dots v_{l-1}v_lp_{v_l}$ forms an $(l+1)$-cap.  
\begin{figure} [h]
\begin{center}
\epsfbox{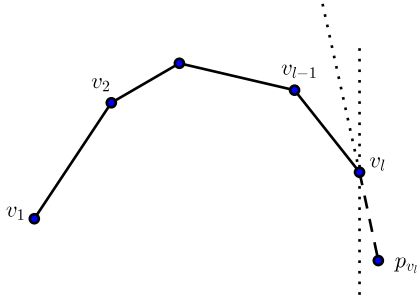}
\end{center}
\caption{Extension of $l$-cap $v_1v_2\dots v_l$ by $p_{v_l}$.}
\label{lemmaext}
\end{figure}

\end{proof}

Now we state the first upper bound which was obtained by Erd\H{o}s and Szekeres \cite{ES35}. First note that in the definition of the function $ES(n)$ given in Theorem \ref{main}, we can simply assume that the set is non-vertical, since by applying an appropriate rotation, one can make the set non-vertical, by also preserving the convex subsets. 

\begin{definition} [$(m,l)$-Free set]
For the natural numbers $m,l \ge 3$, we call a set \emph{$(m,l)$-free} if it contains no $m$-cup and no $l$-cap.
\end{definition}

For the following theorem, we give a modification of the original proof, by using the upper and lower subsets.
\begin{theorem} [Erd\H{o}s and Szekeres \cite{ES35}] \label{ES35} \label{free}
Consider the natural numbers $m,l$ and let $f(m,l)$ denote the maximum natural number such that there exists an $(m,l)$-free set of $f(m,l)$ points in the plane. Then we have $f(m,l)={m+l-4 \choose l-2}$.
\end{theorem}
\begin{proof}
We only prove $f(m,l) \le {m+l-4 \choose l-2}$, since this is the part that we need for the rest of the paper. For this, it is enough to prove 
$$f(m+1,l+1) \le f(m+1,l)+f(m,l+1),$$
since if it holds for all $m,l$, one can apply induction to deduce $f(m,l) \le {m+l-4 \choose l-2}$. 

To prove this, we consider a set $S$ with more than $f(m+1,l)+f(m,l+1)$ points in the plane and prove that $S$ must contain either an $(m+1)$-cup or an $(l+1)$-cap. Let $U_S, L_S$ be the associated upper and lower subsets of $S$, respectively. By Lemma \ref{part}, we have
$$|S|=|U_S|+|L_S| > f(m+1,l)+f(m,l+1),$$
so we either have $|U_S| > f(m+1,l)$ or $|L_S| > f(m,l+1)$. Without loss of generality, assume that $|U_S| > f(m+1,l)$. By the definition, $U_S$ contains either an $(m+1)$-cup or an $l$-cap. In the former case, we are done. In the latter one, by Lemma \ref{ext}, we obtain an $(l+1)$-cap in $S$. So we are done in both cases and the theorem is proved.
\end{proof}
\begin{corollary} \label{corES35}
We have $ES(n) \le {2n-4 \choose n-2}+1$.
\end{corollary}
\begin{proof} Since $n$-cups and $n$-caps are convex $n$-gons, we have $ES(n) \le f(n,n)+1$. So the result is deduced by Theorem \ref{ES35}.
\end{proof}

Now we give a short motivation for the rest of the paper. Consider the point set $S$ with $s \in S$ satisfying the following properties: 
\begin{itemize}
\item [1.] $s$ is the right endpoint of an $(n-2)$-cup in $S$;
\item [2.] $s$ is the left endpoint of an $(n-2)$-cap in $S$;
\item [3.] $s$ is the left endpoint of an $(n-1)$-cup in $S$;
\item [4.] The right endpoint of the $(n-1)$-cup does not coincide with the second point, from the left, of the $(n-2)$-cap. 
\end{itemize}
Then it can be proven that $S$ contains a convex $n$-gon or an $(n-1)$-cap. \\
\\
So a set that contains no convex $n$-gon and no $(n-1)$-cap does not contain any point with all the above 4 properties. Using this fact, for $(n,n-1)$-free sets that contain no convex $n$-gon, we can improve the upper bound on $ES(n)$ which Corollary \ref{corES35} gives. \\
As a result, for any given positive integer $n$, we show the following: Any large enough set of points contains either a convex $n$-gon, or an $(n-1)$-cap or a point $s$ with the above 4 properties. We construct such a point $s$ inductively for sets of large size. So the cardinality of any $(n,n-1)$-free set that contains no convex $n$-gon must be small enough that our inductive procedure does not produce such a point $s$. Therefore, we bound the size of all $(n,n-1)$-free sets that contain no convex $n$-gon. Next, we combine this bound with a result of T\'{o}th and Valtr \cite{TV98} to get the improved bound for $ES(n)$.


\section{Properties of the convexification function}

The aim of this section is to obtain an upper bound on the convexification function (Definition \ref{goodf}). For this, we first prove some of its most important properties. We start with the definitions of convexifying point and convexification function.

\begin{definition} [$(m,l)$-Convexifying point] \label{good}
 Let $m, l \ge 3$ be natural numbers. Consider a set of points $S$ and a point $s \in S$ which is the leftmost point of an $(l-1)$-cap in $S$. We call $s$ an \emph{$(m, l)$-convexifying point} for $S$, if for arbitrary $n \ge 4$, the following holds. For any set $B$ with $B \cap S=\emptyset$, if $B \cup S$ contains an $(n-1)$-cup whose left endpoint is s and whose right endpoint is in $B$, then $B \cup S$ must contain at least one of the followings:
\begin{itemize}
\item[1.] An $m$-cup whose two leftmost points belong to $S$;
\item[2.] An $l$-cap whose two rightmost points belong to $S$;
\item[3.] A convex $n$-gon. 
\end{itemize}

\end{definition} 

\begin{definition} [Convexification function $h(m,l)$] \label{goodf}
For the natural numbers $m,l \ge 3$, define $h(m,l)$ to be the maximum number such that there exists an $(m,l)$-free set of $h(m,l)$ points in the plane with no $(m,l)$-convexifying points. The function $h(m,l)$ is called the \emph{convexification function}.
\end{definition}

\begin{lemma} [Subadditivity] \label{subadd}
The convexification function is subadditive. In other words, for any natural numbers $m,l \ge 3$, we have
$$h(m+1,l+1) \le h(m+1,l)+h(m,l+1).$$
\end{lemma}
\begin{proof}
Consider the $(m+1,l+1)$-free set $S$ containing more than $h(m+1,l) + h(m,l+1)$ points in the plane. We have to show that $S$ has an $(m+1,l+1)$-convexifying point. Let $U_S, L_S$ be the upper and lower subsets of $S$, respectively. Since by Lemma \ref{part} we have
$$|S|=|U_S|+|L_S| > h(m+1,l) + h(m,l+1),$$
we get either $|U_S| > h(m + 1, l)$ or $|L_S| > h(m, l+ 1)$. 

Consider the former case. Note that by Lemma \ref{ext}, $U_S$ is an $(m+1,l)$-free set, since $S$ is $(m+1,l+1)$-free. So $U_S$ has an $(m+1,l)$-convexifying point. Denote this point by $s$. We show that $s$ is also an $(m+1,l+1)$-convexifying point for $S$.  For this, first note that $s$ is the leftmost point of an $(l-1)$-cap in $U_S$. This cap can be extended to an $l$-cap in $S$ from right, by Lemma \ref{ext}, so $s$ is the left endpoint of an $l$-cap in $S$. Now let $n \ge 4$ be an arbitrary natural number and consider the set $B$ with $B \cap S=\emptyset$, such that $B \cup S$ contains an $(n-1)$-cup whose left endpoint is $s$ and whose right endpoint is in $B$. Call this $(n-1)$-cup $C$ and construct the set $B'$ as 
$$B'= (B \cup L_S) \cap C=C \setminus U_S.$$
Clearly, $B' \cap U_S=\emptyset$, and $B' \cup U_S$ contains the $(n-1)$-cup $C$, whose left endpoint is s and whose right endpoint is in $B'$. So by the Definition \ref{good} and $(m+1,l)$-convexifying property of s for $U_S$, $B' \cup U_S$ must contain at least one of the following:
\begin{enumerate}
\item [1.] An $(m+1)$-cup whose two leftmost points belong to $U_S$;
\item [2.] An $l$-cap whose two rightmost points belong to $U_S$;
\item [3.] A convex $n$-gon.
\end{enumerate}

Since $B' \cup U_S \subset B \cup S$ and $U_S \subset S$, in the cases 1 and 3, we get a desired $(m+1)$-cup and a convex $n$-gon in $B \cup S$, respectively, which makes $s$ an $(m+1,l+1)$-convexifying point for $S$. Now consider case 2 and call this $l$-cap $C'$. Since two rightmost points of $C'$ belong to $U_S$, by Lemma \ref{ext}, $C'$ can be extended to an $(l+1)$-cap by adding an appropriate point of $L_S$ to its right. This way, we get an $(l+1)$-cap in $B \cup S$ whose two rightmost points belong to $S$, which makes $s$ an $(m+1,l+1)$-convexifying point for $S$.
So $s$ is a convexifying point for $S$ in all of these three cases and we are done. The proof of the second case, where $|L| > h(m, l+ 1)$, is exactly similar to this one. 
\end{proof}

\begin{theorem} \label{m,4}
For any natural number $m \ge 4$, we have $h(m,4) \le {m-1 \choose 2}+1$.
\end{theorem}
\begin{proof} The proof consists of two parts. First, we prove that for the $(m,4)$-free set $S$, if it contains an $(m-1)$-cup with a point lying to the right of it, then the second rightmost point of this cup is an $(m,4)$-convexifying point of $S$. Secondly, we prove that if $|S| > {m-1 \choose 2}+1$, then $S$ contains such a cup.

Consider the first part and assume $S$ contains the $(m-1)$-cup $v_1v_2\dots v_{m-1}$, which is ordered in increasing order of $x$-coordinate, and consider $r \in S$ lying to the right of $v_{m-1}$. Note that $r$ must lie below the line spanned by $v_{m-2}v_{m-1}$, otherwise we get an $m$-cup which is in contradiction with the $(m,4)$-free property of $S$. We show that $v_{m-2}$ is an $(m,4)$-convexifying point for $S$. For this, note that $v_{m-2}v_{m-1}r$ is a $3$-cap, so $v_{m-2}$ is the leftmost point of a $3$-cap in $S$. Now take an arbitrary natural number $n \ge 4$ and consider the set $B$ with $B \cap S=\emptyset$, such that $B \cup S$ contains the $(n-1)$-cup $v_{m-2}u_2u_3\dots u_{n-1}$, again ordered in increasing order of $x$-coordinate, where $u_{n-1} \in B$. So in particular $v_{m-1} \neq u_{n-1}$. We split the rest of the proof into the following cases:
\\
\item []\textbf{Case a:} \emph{$u_{n-1}$ is to the left of $v_{m-1}$ and above the line $v_{m-2}v_{m-1}$:}
(Note that in this case we have $v_{m-1} \notin \{u_2,u_3,\dots,u_{n-2}\}$.)
\begin{itemize}
\item [] \textbf{Subcase a1:} \emph{$u_{n-2}$ is above the line $v_{m-3}v_{m-2}$:}
\\ $v_1v_2\dots v_{m-2}u_{n-2}u_{n-1}$ forms an $m$-cup with two leftmost points in $S$. This holds since both $v_{m-3}v_{m-2}u_{n-2}$ and $v_{m-2}u_{n-2}u_{n-1}$ form $3$-cups.
\item[] \textbf{Subcase a2:} \emph{$u_{n-2}$ is below the line $v_{m-3}v_{m-2}$ and $v_{m-1}$ is above the line $u_{n-3}u_{n-2}$:} \\$v_{m-2}u_2u_3\dots u_{n-2}v_{m-1}u_{n-1}$ forms a convex $n$-gon. This holds since $v_{m-2}u_2u_3\dots u_{n-1}$ forms a convex $(n-1)$-gon and $v_{m-1}$ is to the right of $u_{n-1}$, lies below $v_{m-2}u_{n-1}$ and above $u_{n-3}u_{n-2}$.
\item[] \textbf{Subcase a3:} \emph{$u_{n-2}$ is below the line $v_{m-3}v_{m-2}$ and $v_{m-1}$ is below the line $u_{n-3}u_{n-2}$:} 
\\ $u_{n-3}u_{n-2}v_{m-1}r$ forms a $4$-cap with two rightmost points in $S$. This holds since $u_{n-2}$ is below $v_{m-3}v_{m-2}$, so because it is between $v_{m-2}$ and $v_{m-1}$, it is also below $v_{m-2}v_{m-1}$. On the other hand, $r$ is also below $v_{m-2}v_{m-1}$, by the assumption. So $u_{n-2}v_{m-1}r$ forms a $3$-cap. $u_{n-3}u_{n-2}v_{m-1}$ also forms a $3$-cap, as well.
\end{itemize}  
\begin{figure} [h!]
\begin{center}
\epsfbox{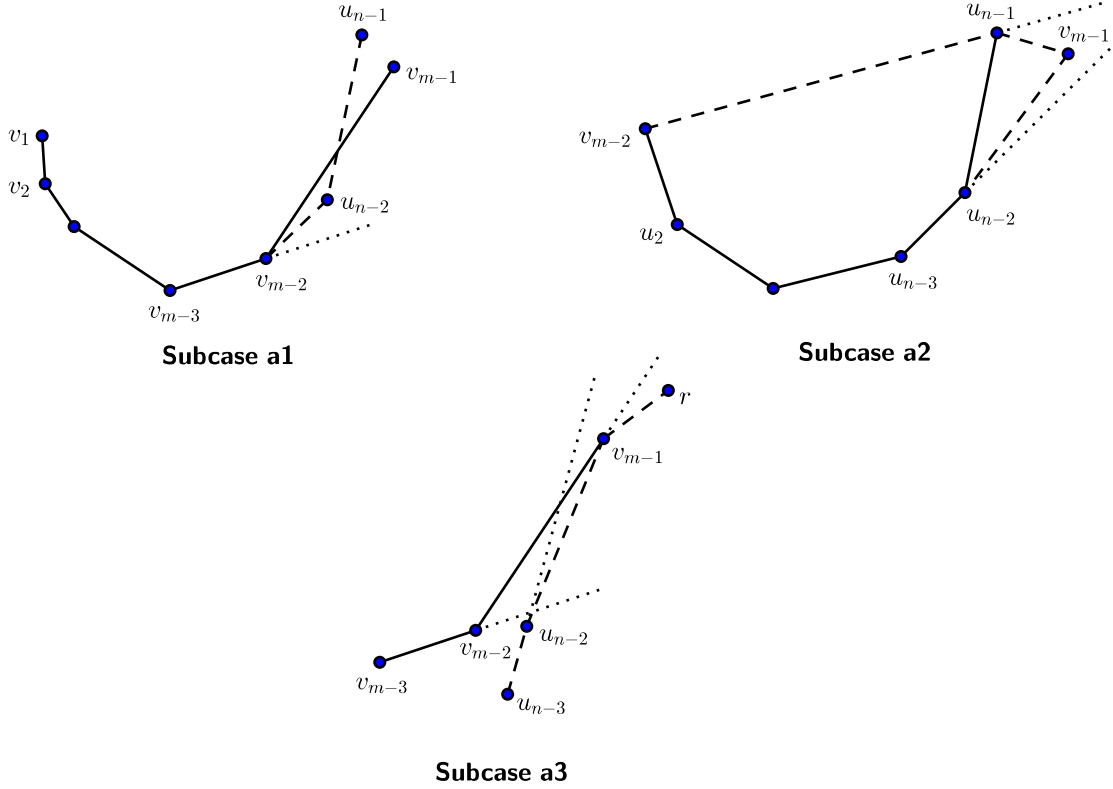}
\end{center}
\caption{Case a in Theorem \ref{m,4}.}
\label{case a}
\end{figure}
\item[] \textbf{Case b:} \emph{$u_{n-1}$ is to the left of $v_{m-1}$ and below the line $v_{m-2}v_{m-1}$}:
\begin{itemize}
\item[] \textbf{Subcase b1:} \emph{$v_{m-1}$ is above the line $u_{n-2}u_{n-1}$:} 
\\ $v_{m-2}u_2u_3\dots u_{n-1}v_{m-1}$ forms an $n$-cup, and as a result, a convex $n$-gon.
\item[] \textbf{Subcase b2:} \emph{$v_{m-1}$ is below the line $u_{n-2}u_{n-1}$:}
\\ $u_{n-2}u_{n-1}v_{m-1}r$ forms a $4$-cap with two rightmost points in $S$. This holds since both $u_{n-1}$ and $r$ are below $v_{m-2}v_{m-1}$, $u_{n-1}v_{m-1}r$ forms a $3$-cap. $u_{n-2}u_{n-1}v_{m-1}$ also forms a $3$-cap.
\end{itemize}  
\begin{figure} 
\begin{center}
\epsfbox{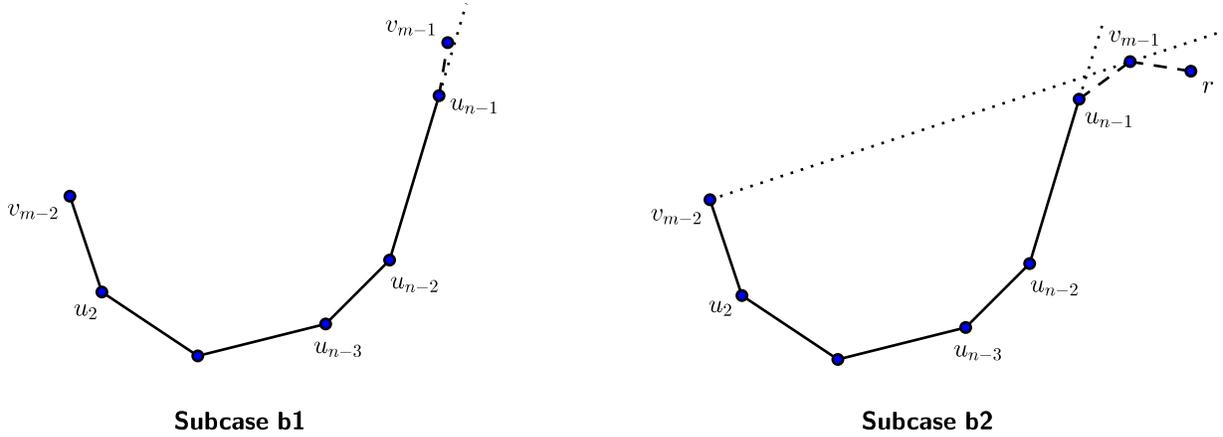}
\end{center}
\caption{Case b in Theorem \ref{m,4}.}
\label{case b}
\end{figure}
\item[] \textbf{Case c:} \emph{$u_{n-1}$ is to the right of $v_{m-1}$ and above the line $v_{m-2}v_{m-1}$:} 
\\ $v_1v_2\dots v_{m-2}v_{m-1}u_{n-1}$ forms an $m$-cup with two leftmost points in $S$.
\\
\item[] \textbf{Case d:} \emph{$u_{n-1}$ is to the right of $v_{m-1}$ and below the line $v_{m-2}v_{m-1}$:} 
\\ $v_{m-2}u_2u_3\dots u_{n-1}v_{m-1}$ forms a convex $n$-gon. This holds by following the same reasoning as subcase a2. Just note that $u_{n-1}$ is below the line $v_{m-2}v_{m-1}$. Also, all the points $u_2,u_3,\dots,u_{n-2}$ are below the line $v_{m-2}u_{n-1}$, so we have $v_{m-1} \notin \{u_2,u_3,\dots,u_{n-2}\}$.
\begin{figure} [h]
\begin{center}
\epsfbox{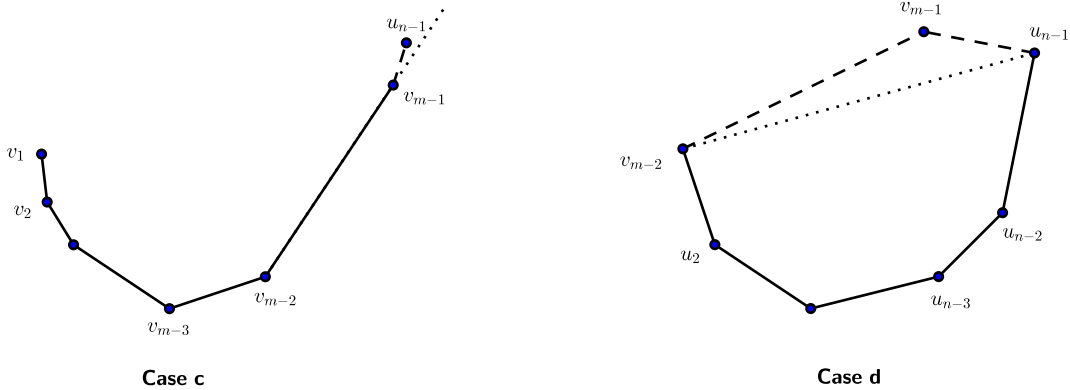}
\end{center}
\caption{Cases c, d in Theorem \ref{m,4}.}
\label{cases c,d}
\end{figure}
\\
So, according to the Definition \ref{good}, $v_{m-2}$ is an $(m,4)$-convexifying point for $S$. 

Now we prove the second part. Suppose $S$ is an $(m,4)$-free set with more than ${m-1 \choose 2}+1$ points in the plane. Let $r$ be the point of $S$ with the largest $x$-coordinate. Consider the set $S\setminus\{r\}$. We have 
$$|S\setminus\{r\}| > {m-1 \choose 2}=f(m-1,4),$$
so by the Theorem \ref{ES35}, $S\setminus\{r\}$ must contain either an $(m-1)$-cup or a $4$-cap. The latter cannot happen since $S$ is $(m,4)$-free, so $S\setminus\{r\}$ contains an $(m-1)$-cup. By the choice of $r$, it must lie to the right of the rightmost point of this cup, so we are done. This completes the proof of the theorem.
\end{proof}

The following theorem gives an upper bound for $h(4,l)$.

\begin{theorem} \label{4,l}
For any natural number $l \ge 4$, we have $h(4,l) \le {l-1 \choose 2}+1$. 
\end{theorem}
\begin{proof} We split the proof into two parts as in the proof of Theorem \ref{m,4}. First we prove that for the $(4,l)$-free set $S$, if it contains an $(l-1)$-cap with a point lying to the left of it, then the leftmost point of this cap is an $(4,l)$-convexifying point of $S$. Secondly, we prove that if $|S| > {l-1 \choose 2}+1$, then $S$ contains such a cap. 

We proceed with the first part. First note that we cannot deduce this part from Theorem \ref{m,4} by symmetry, since in both of them, we are dealing with $(n-1)$-cups. Assume $S$ contains the $(l-1)$-cap $v_1v_2\dots v_{l-1}$, which is ordered in increasing order of $x$-coordinate, and consider $r \in S$ lying to the left of $v_1$. Note that $r$ must lie above the line spanned by $v_1v_2$, otherwise we get an $l$-cap which is in contradiction with the $(4,l)$-free property of $S$. We show that $v_1$ is an $(4,l)$-convexifying point for $S$. Evidently, $v_1$ is the left endpoint of an $(l-1)$-cap in $S$. Now take an arbitrary natural number $n \ge 4$ and consider the set $B$ with $B \cap S=\emptyset$, such that $B \cup S$ contains the $(n-1)$-cup $v_1u_2u_3\dots u_{n-1}$, again ordered in increasing order of $x$-coordinate, where $u_{n-1} \in B$. Note that if there exists $2\le i \le n-3$ such that $u_i \in S$, then $v_1u_iu_{n-2}u_{n-1}$ forms a $4$-cup with the two leftmost points belonging to $S$, which means $v_1$ is an $(4,l)$-convexifying point for $S$. On the other hand, if we have $u_{n-2}=v_2$, then $rv_1v_2u_{n-1}$ forms a $4$-cup with the required property.

So we can assume that $u_i \in B$ for all $i=2,\dots,n-3$, and $u_{n-2} \neq v_2$. Now we split the rest of the proof into the following cases:
\\
\item[] \textbf{Case a:} \emph{$u_{n-1}$ is to the left of $v_2$ and above the line $v_1v_2$:}
\begin{itemize} 
\item[] \textbf{Subcase a1:} \emph{$u_{n-2}$ is above the line $v_1v_2$:} 
\\ $rv_1u_{n-2}u_{n-1}$ forms a $4$-cup with two leftmost points in $S$. This holds since both of $r, u_{n-2}$ lie above $v_1v_2$, so $rv_1u_{n-2}$ forms a $3$-cup. $v_1u_{n-2}u_{n-1}$ also forms a $3$-cup, by the assumption.
\item [] \textbf{Subcase a2:} \emph{$u_{n-2}$ is below the line $v_1v_2$ and $v_2$ is above the line $u_{n-3}u_{n-2}$:} 
\\ $v_1u_2u_3\dots u_{n-2} v_2u_{n-1}$ forms a convex $n$-gon. This holds since $v_1u_2u_3\dots u_{n-1}$ forms a convex $(n-1)$-gon, and $v_2$ is below both of $v_1u_{n-1}$ and $u_{n-2}u_{n-1}$, and above $u_{n-3}u_{n-2}$.
\item[] \textbf{Subcase a3:} \emph{$u_{n-2}$ is below the line $v_1v_2$ and $v_2$ is below the line $u_{n-3}u_{n-2}$:} 
\\ $u_{n-3}u_{n-2}v_2v_3\dots v_{l-1}$ forms an $l$-cap with two rightmost points in $S$. First note that $u_{n-2} \notin \{v_3,\dots,v_{l-1}\}$, since $x_{u_{n-2}} < x_{u_{n-1}} < x_{v_2} <  x_{v_i}$, for all $i\ge 3$. It is enough to show that $u_{n-3}u_{n-2}v_2$ and $u_{n-2}v_2v_3$ form $3$-caps. The first one is a $3$-cap, since $v_2$ is below $u_{n-3}u_{n-2}$. For the second one, since both of $u_{n-2}$ and $v_3$ are below $v_1v_2$, we get a $3$-cap, too.
\end{itemize} 
\begin{figure} [h]
\begin{center}
\epsfbox{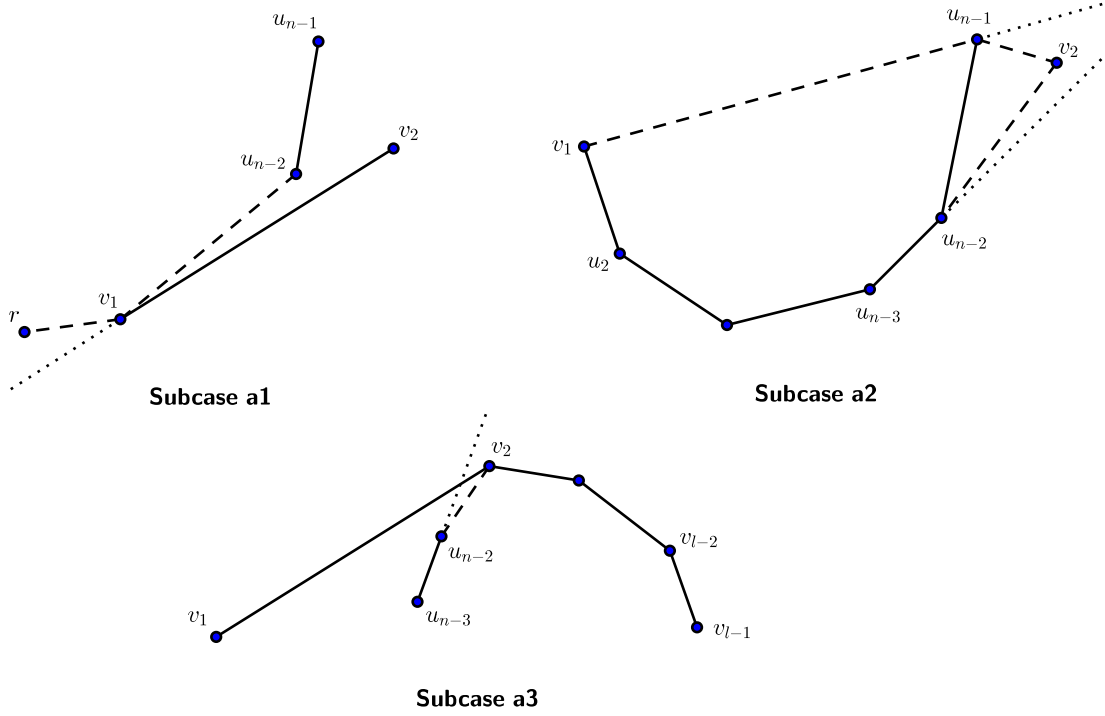}
\end{center}
\caption{Case a in Theorem \ref{4,l}.}
\label{case a2}
\end{figure}
\item[] \textbf{Case b:} \emph{$u_{n-1}$ is to the left of $v_2$ and below the line $v_1v_2$:} 
\begin{itemize} 
\item[] \textbf{Subcase b1:} \emph{$v_2$ is above the line $u_{n-2}u_{n-1}$:} 
\\ $v_1u_2u_3\dots u_{n-1}v_2$ forms an $n$-cup, and as a result, a convex $n$-gon.
\item[] \textbf{Subcase b2:} \emph{$v_2$ is below the line $u_{n-2}u_{n-1}$:} 
\\ $u_{n-2}u_{n-1}v_2v_3\dots v_{l-1}$ forms an $l$-cap with two rightmost points in $S$. First note that with a similar reasoning as Subcase a3, we have $u_{n-2} \notin \{v_3,\dots,v_{l-1}\}$. So it is enough to show that $u_{n-2}u_{n-1}v_2$ and $u_{n-1}v_2v_3$ form $3$-caps. The first one is a $3$-cap since $v_2$ is below $u_{n-2}u_{n-1}$. For the second one, since $u_{n-1}$ and $v_3$ are both below $v_1v_2$, $u_{n-1}v_2v_3$ must form a $3$-cap.
\end{itemize} 
\begin{figure} [h!]
\begin{center}
\epsfbox{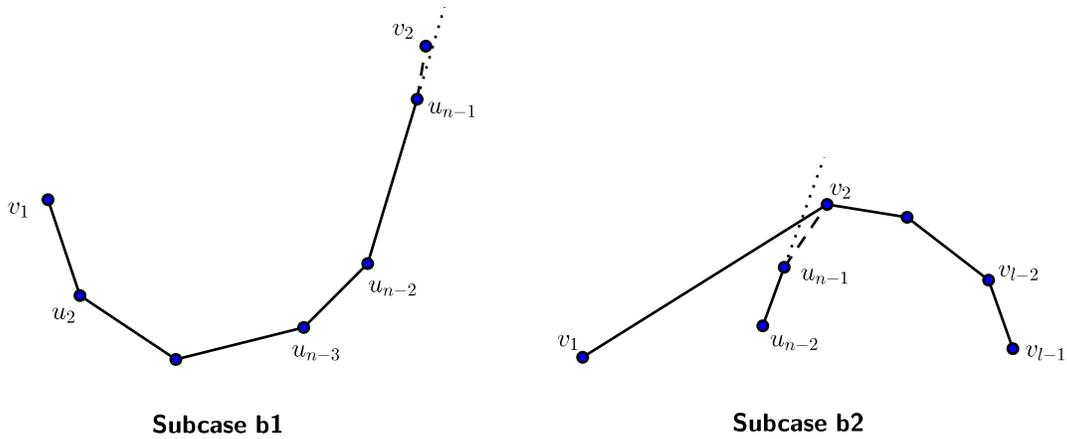}
\end{center}
\caption{Case b in Theorem \ref{4,l}.}
\label{case b2}
\end{figure}
\item[] \textbf{Case c:} \emph{$u_{n-1}$ is to the right of $v_2$ and above the line $v_1v_2$:} 
\\ $rv_1v_2u_{n-1}$ forms a $4$-cup with two leftmost points in $S$. This holds because $rv_1v_2$ forms a $3$-cup, since $r$ is above $v_1v_2$. $u_{n-1}$ is above $v_1v_2$, so $v_1v_2u_{n-1}$ also forms a $3$-cup.
\\
\item[] \textbf{Case d:} \emph{$u_{n-1}$ is to the right of $v_2$ and below the line $v_1v_2$:} 
\\ $v_1u_2u_3\dots u_{n-1}v_2$ forms a convex $n$-gon. This holds since $v_2$ is to the right of $v_1$, to the left of $u_{n-1}$ and above $v_1u_{n-1}$.
\\
So, according to the Definition \ref{good}, $v_1$ is an $(4,l)$-convexifying point for $S$.

The second part can be proved in the same way as the second part in the Theorem \ref{m,4}, either by following that proof or using that result with symmetry in order to replace cups with caps. 
\begin{figure} [h]
\begin{center}
\epsfbox{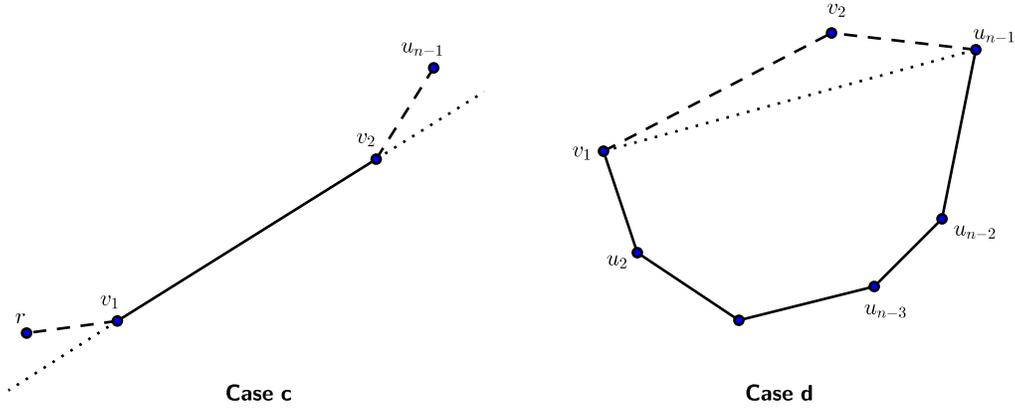}
\end{center}
\caption{Cases c, d in Theorem \ref{4,l}.}
\label{casescd2}
\end{figure}

\end{proof}

Now we combine the Theorems \ref{m,4}, \ref{4,l} with Lemma \ref{subadd} in order to get the following bound on the convexification function $h(m,l)$.
\begin{theorem} \label{boundgood}
For any natural numbers $m,l \ge 4$, we have
$$h(m,l) \le {m+l-4 \choose l-2}-{m+l-6 \choose l-3}.$$

\end{theorem}
\begin{proof}
First define the function $g(m,l)$ as
$$g(m,l)= {m+l-4 \choose l-2}-{m+l-6 \choose l-3}.$$
So we need to show that $h(m,l) \le g(m,l)$, for all $m,l\ge 4$. We apply double induction on $m$ and $l$. For the induction basis, suppose $m=4$. Then according to Theorem \ref{4,l} we have
$$h(4,l) \le {l-1 \choose 2}+1={l \choose 2}-{l-2 \choose 1}=g(4,l).$$
Similarly, the other induction basis $l=4$ can be proved by using Theorem \ref{m,4}. Now suppose the inequality holds for $(m+1,l)$ and $(m,l+1)$, we prove it for $(m+1,l+1)$. Note that based on the identity
$${s \choose t}={s-1 \choose t}+{s-1 \choose t-1},$$
we can get that 
$$g(m+1,l+1)=g(m+1,l)+g(m,l+1).$$
By Lemma \ref{subadd} we have
$$h(m+1,l+1) \le h(m+1,l)+h(m,l+1) \le g(m+1,l)+g(m,l+1)=g(m+1,l+1),$$
where the last inequality holds according to the induction hypothesis. This completes the proof.
\end{proof}

\section{Back to the main problem}

In this section, we prove the main result of this paper, Theorem \ref{main}. We start with the following lemma which we need for Theorem \ref{f+g}.

\begin{lemma}[\cite{ES35}] \label{extt} If a point is the rightmost point of a cup (cap) and also the leftmost point of a cap (cup), then the cup or the cap can be extended to a larger cup or cap, respectively.
\end{lemma}
\begin{proof}
Denote the cup by $v_1v_2\dots v_m$ and the cap by $v_mu_2\dots u_l$. We split the rest of the proof into the following two cases. See Figure \ref{exttfig}.
\item[] \textbf{Case a:} \emph{$u_2$ lies above the line spanned by $v_{m-1}v_m$:} $v_1v_2\dots v_mu_2$ forms an $(m+1)$-cup.
\item[] \textbf{Case b:} \emph{$u_2$ lies below the line spanned by $v_{m-1}v_m$:} $v_{m-1}v_mu_2\dots u_l$ forms an $(l+1)$-cap.

\begin{figure} [h]
\begin{center}
\epsfbox{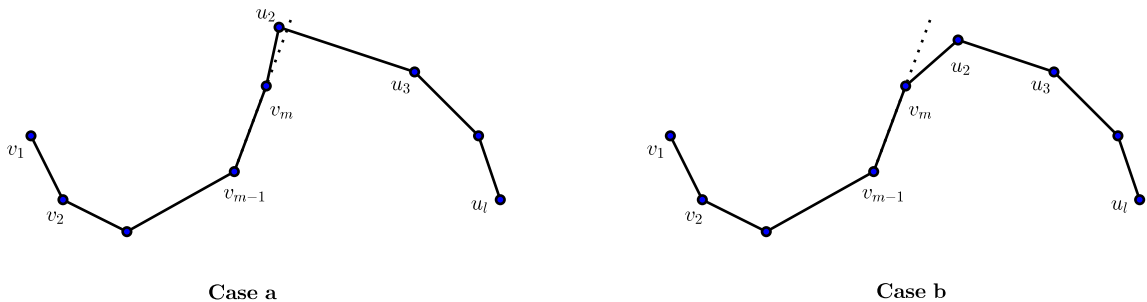}
\end{center}
\caption{Case by case analysis in Lemma \ref{extt}.}
\label{exttfig}
\end{figure}

\end{proof}

\begin{theorem} \label{f+g}
Let $n \ge 6$ be a natural number and consider the set $S$ of $f(n-1,n-1)+g(n,n-2)+1$ points in the plane, where the function $g$ is defined in Theorem \ref{boundgood}. Then $S$ contains an $(n-1)$-cap or a convex $n$-gon.
\end{theorem}

\begin{proof}
Consider the partition of $S$ into the upper and lower subsets $U_S,L_S$ as in the Definition \ref{UL}. If $|L_S| > f(n-1,n-1)$, then it either contains an $(n-1)$-cup or an $(n-1)$-cap. In the former case, we get an $n$-cup by the Lemma \ref{ext}, and in the latter case, we are immediately done. So we can assume that $|L_S| \le f(n-1,n-1)$. On the other hand, we can assume that $U_S$ is $(n,n-2)$-free, since otherwise we are immediately done or again by Lemma \ref{ext}. We have
$$|U_S|=|S|-|L_S|=g(n,n-2)+(f(n-1,n-1)+1-|L_S|),$$ 
where the last term is positive. So by applying Theorem \ref{boundgood} iteratively, we get $f(n-1,n-1)+1-|L_S|$ $(n,n-2)$-convexifying points for $U_S$. Denote the set of these convexifying points by $G$. Consider the set $G \cup L_S$. We have 
$$|G \cup L_S|=f(n-1,n-1)+1,$$
so it must contain either an $(n-1)$-cup or an $(n-1)$-cap. In the latter case, we are done. In the former case, denote this $(n-1)$-cup by $v_1v_2\dots v_{n-1}$. If $v_1 \in L_S$, by the Lemma \ref{ext} we get an $n$-cup in $S$, which means we are done. On the other hand, every point of $G$ is the leftmost point of an $(n-3)$-cap in $U_S$, which by Lemma \ref{ext}, can be extended to an $(n-2)$-cap in $S$ by adding a point to the right of it. So if $v_{n-1} \in G$, then by Lemma \ref{extt}, we either get an $n$-cup or an $(n-1)$-cap, which finishes the proof. 

So we can assume that $v_1 \in G$ and $v_{n-1} \in L_S$. Now by Definition \ref{good}, we conclude that $S$ contains at least one of the followings:
\begin{itemize}
\item [1.] An $n$-cup;
\item [2.] An $(n-2)$-cap whose two rightmost points belong to $U_S$;
\item [3.] A convex $n$-gon. 
\end{itemize}
In the cases 1, 3, we are done. In case 2, according to Lemma \ref{ext}, we get an $(n-1)$-cap, so we are also done in this case. This completes the proof of the theorem.
\end{proof}

The proof of the following theorem is completely based on and similar to the proof of Theorem 5 in \cite{TV03}.
\begin{theorem} [\cite{TV98}, \cite{TV03}] \label{PT}
Define $p(n)$ to be the minimum natural number such that every set with at least $p(n)$ points in the plane contains either an $(n-1)$-cap or a convex $n$-gon. Then we have 
$$ES(n) \le p(n)+1.$$
\end{theorem}
\begin{proof}
Let $S$ be a set of $p(n)+1$ points in the plane. We have to show that $S$ contains $n$ points in convex position. Let $s$ be a vertex of the convex hull of $S$, $\mathrm{conv}(S)$, and take the point $s'$ outside $\mathrm{conv}(S)$ such that none of the lines spanned by the points of $S\setminus \{s\}$ intersects the segment $ss'$. Also, take a line $l$ through $s'$ which does not intersect $\mathrm{conv}(S)$. 

Now consider the projective transformation taking the line $l$ to the line at infinity and also mapping the segment $ss'$ to the vertical half-line emanating downwards from $T(s)$. Since the line $l$ avoided $\mathrm{conv}(S)$, $T$ does not change convexity on the points of $S$, i.e. $P \subset S$ is in convex position if and only if $T(P) \subset T(S)$ is in convex position. 

One can see that none of the lines spanned by the points of $T(S) \setminus \{T(s)\}$ intersects the vertical half-line emanating downwards from $T(s)$. Furthermore, $T(S) \setminus \{T(s)\}$ is a non-vertical set. So for any natural number $l$, any $l$-cap in $T(S) \setminus \{T(s)\}$ can be extended to a convex $(l+1)$-gon by adding the point $T(s)$. Therefore, because $|S|=p(n)+1$, we have $|T(S) \setminus \{T(s)\}|=p(n)$, and by the definition of the function $p(n)$, we get that $T(S) \setminus \{T(s)\}$ must contain either an $(n-1)$-cap or a convex $n$-gon. Based on what was stated before, $T(S)$ contains a convex $n$-gon, and as a result, $S$ contains a convex $n$-gon, as well. This completes the proof.

\end{proof}

Finally, we prove Theorem \ref{main}, and use it to obtain the asymptotic upper bound \eqref{eq3}.
\begin{proof}[Proof of Theorem \ref{main}] Recall the function $p(n)$ from Theorem \ref{PT}. By Theorem \ref{f+g} we get
\begin{align*}
p(n) \le f(n-1,n-1)+g(n,n-2)+1&={2n-6 \choose n-3}+{2n-6 \choose n-2}-{2n-8 \choose n-3}+1 \\
& ={2n-5 \choose n-2}-{2n-8 \choose n-3}+1.
\end{align*}
Combining the above with 
$$ES(n) \le p(n)+1$$
from Theorem \ref{PT}, we get the desired bound.
\end{proof}
\begin{corollary} We have
$$\limsup\limits_{n\rightarrow\infty} \frac{ES(n)}{{2n-5 \choose n-2}} \le \frac{7}{8}.$$
\end{corollary} 
\begin{proof}
By Theorem \ref{main}, it is enough to show that 
$$ \limsup\limits_{n\rightarrow\infty} \frac{{2n-5 \choose n-2}-{2n-8 \choose n-3}+2}{{2n-5 \choose n-2}}=\frac{7}{8}.$$
But we have
$$ \limsup\limits_{n\rightarrow\infty} \frac{{2n-5 \choose n-2}-{2n-8 \choose n-3}+2}{{2n-5 \choose n-2}}=1+\limsup\limits_{n\rightarrow\infty} \frac{-{2n-8 \choose n-3}}{{2n-5 \choose n-2}}=1-\liminf\limits_{n\rightarrow\infty} \frac {{2n-8 \choose n-3}}{{2n-5 \choose n-2}}=1-\frac{1}{8}=\frac{7}{8}.$$

\end{proof}


\begin{thebibliography}{99}

\bibitem{BMP}
P. Brass, W. Moser, and J. Pach, 
\emph{Convex polygons and the Erd\H{o}s-Szekeres problem}, 
Chapter 8.2 in the book \emph{Research problems in discrete geometry}, Springer (2005).

\bibitem{CG}
F. R. K. Chung and R. L. Graham, 
\emph{Forced convex $n$-gons in the plane}, 
Discrete and Computational Geometry \textbf{19} (1998), 367-371.

\bibitem{ES35}
P. Erd\H{o}s and G. Szekeres,
\emph{A combinatorial problem in geometry},
Compositio Mathematica \textbf{2} (1935), 463-470.

\bibitem{ES61}
P. Erd\H{o}s and G. Szekeres,
 \emph{On some extremum problems in elementary geometry}, 
 Ann. Universitatis Scientiarum Budapestinensis, E\"{o}tv\"{o}s, Sectio Mathematica \textbf{3/4} (1960-61), 53-62.

\bibitem{KP}
D. J. Kleitman and L. Pachter, 
\emph{Finding convex sets among points in the plane}, 
Discrete and Computational Geometry \textbf{19} (1998), 405-410.
 
 \bibitem{NY}
 S. Norin and Y. Yuditsky,
\emph{Erd\H{o}s-Szekeres without induction},
arXiv:1509.03332 (2015).
 
\bibitem{TV98}
G. T\'{o}th and P. Valtr,
\emph{Note on the Erd\H{o}s-Szekeres theorem}, 
Discrete and Computational Geometry \textbf{19} (1998), 457-459.

 \bibitem{TV03}
G. T\'{o}th and P. Valtr,
\emph{The Erd\H{o}s-Szekeres theorem: Upper bounds and related results},
Combinatorial and Computational Geometry (J.E. Goodman et al., eds.), Publ. M.S.R.I. \textbf{52} (2006), 557-568.

\bibitem{V}
G. Vlachos,
\emph{On a conjecture of Erd\H{o}s and Szekeres}, 
arXiv:1505.07549 (2015).


\end{thebibliography}
\end{document}